\documentclass[12pt,a4paper]{amsart}
\usepackage[top=35mm, bottom=35mm, left=30mm, right=30mm]{geometry}
\usepackage{mathptmx}
\usepackage{mathrsfs}
\usepackage{verbatim}
\usepackage{url}
\usepackage[all]{xy}
\usepackage{xcolor}
\usepackage[colorlinks=true,citecolor=blue]{hyperref}

\theoremstyle{plain}

\newtheorem{thm}{Theorem}[section]
\newtheorem{lem}[thm]{Lemma}

\theoremstyle{definition}
\newtheorem{defn}[thm]{Definition}

\newcommand{\Z}{{\mathbb{Z}_+}}
\newcommand{\N}{\mathbb{N}}

\newcommand{\mZ}{\mathbb Z}

\newcommand{\ep}{\varepsilon}

\DeclareMathOperator{\diam}{diam}
\DeclareMathOperator{\supp}{supp}

\begin{document}
\title{A note on mean equicontinuity}
\author[J.~Qiu]{Jiahao Qiu}
\address[J.~Qiu]{Wu Wen-Tsun Key Laboratory of Mathematics, USTC, Chinese Academy of Sciences and
School of Mathematics, University of Science and Technology of China,
Hefei, Anhui, 230026, P.R. China}
\email{qiujh@mail.ustc.edu.cn}

\author[J.~Zhao ]{Jianjie Zhao}
\address[J.~Zhao]{Wu Wen-Tsun Key Laboratory of Mathematics, USTC, Chinese Academy of Sciences and
School of Mathematics, University of Science and Technology of China,
Hefei, Anhui, 230026, P.R. China}
\email{zjianjie@mail.ustc.edu.cn}

\date{\today}

\begin{abstract}
In this note, it is shown that several results concerning mean equicontinuity proved before for minimal systems are actually held
for general topological dynamical systems. Particularly, it turns out that a dynamical system is mean equicontinuous if and only if it is equicontinuous in the mean
if and only if it is Banach (or Weyl) mean equicontinuous if and only if its regionally proximal relation is equal to the Banach  proximal relation.

Meanwhile,
a relation is introduced such that the smallest closed invariant equivalence relation containing this relation induces the maximal mean equicontinuous factor
for any system.
%we introduce the system with equicontinuity in the mean and give some equivalent conditions.
%We show that for every dynamical system $(X,T)$
%mean equicontinuity is equivalent to equicontinuity in the mean.
%It is well known that every dynamical system $(X,T)$ has the maximal equicontinuous factor,
%we construct the maximal regionally proximal structure in the mean,
%and show that every  dynamical system $(X,T)$ has the maximal mean equicontinuous factor.
\end{abstract}
\maketitle
%\tableofcontents
%%%%%%%%%%%%%%%%%%%%%%%%%%%%%%%%%%%%%%%%%%%%%%%%%%%%%%%%%%%%%
%%%%%%%%%%%%%%%%%%%%%%%%%%%%%%%%%%%%%%%%%%%%%%%%%%%%%%%%%%%%%

\section{Introduction}
Throughout this paper, \emph{a topological dynamical system} is a pair $(X,T)$, where $X$ is a
non-empty compact metric space with a metric $d$ and $T$ is a continuous map from $X$ to
itself.
%Let $T$ be a continuous map of a topological space $X$ (the topology is induced by the metric $d$) into itself.
%We denote the \emph{(topological) dynamical system} by $(X, T).$
%Let $X$ be a compact metric space with a metric $d$, and let $T$ be a continuous map from $X$ to itself.
%The pair $(X,T)$ will be called a \emph{(topological) dynamical system}.

We all know that equicontinuous systems have simple dynamical behaviors.
By the well known Halmos-von Neumann theorem, a transitive equicontinuous system is conjugate to
a minimal rotation on a compact abelian metric group,
and $(X,T,\mu)$ has discrete spectrum, where $\mu$ is the unique Haar measure on $X$.
In this note, we discuss the systems with equicontinuity in the mean sense.
%Recall that a dynamical system $(X,T)$ is called \emph{equicontinuous} if for every $\ep>0$ there is a $\delta>0$ such that
%whenever $x,y\in X$ with $d(x,y)<\delta$, we have $d(T^nx,T^ny)<\ep$ for all $n\in\Z$.
%that is, the family of maps $\{T^n\colon n\in \Z\}$ is uniformly equicontinuous.
%Equicontinuous systems have simple dynamical behaviors.
%It is well known that a dynamical system $(X,T)$ with $T$ being surjective is equicontinuous if and only if
%there exists a compatible metric $\rho$ on $X$ such that $T$ acts on $X$ as an isometry, i.e.,
%$\rho(Tx,Ty)=\rho(x,y)$ for any $x,y\in X$. Moreover,

Recall that a dynamical system $(X,T)$ is called \emph{mean equicontinuous} if for every $\ep>0$,
there exists a $\delta>0$ such that whenever $x,y\in X$ with $d(x,y)<\delta$, we have
\[\limsup_{n\to\infty} \bar d_n(x,y)<\ep,\ \text{where}\ \bar d_n(x,y)=\frac{1}{n}\sum_{i=0}^{n-1}d(T^ix,T^iy).\]

A notion called \emph{stable in the mean in the sense of Lyapunov} or simply \emph{mean-L-stable} is introduced by Fomin~\cite{F51}.
We call a dynamical system $(X,T)$ \emph{mean-L-stable}
if for every $\ep>0$, there is a $\delta>0$ such that $d(x,y)<\delta$
implies $d(T^nx,T^ny)<\ep$ for all $n\in\Z$ except a set of
upper density less than $\ep$.
Oxtoby \cite{O52}, Auslander \cite{A59} and Scarpellini \cite{S82} also studied mean-L-stable systems.
 It is easy to see that a dynamical system
is mean-L-stable if and only if it is mean equicontinuous. Answering an open question in \cite{S82}, it was proved
by Li, Tu and Ye in \cite{YL} that a minimal mean equicontinuous system has discrete spectrum. We refer to \cite{Felipe1, Felipe2, Felipe3, GLZ17, HL, Li} for further study
on mean equicontinuity and related subjects.
%briefly introduced the mean-L-stable systems, and he proved
%that if a mean-L-stable system is transitive then it is uniquely ergodic.

%Mean equicontinuity is discussed by Li, Tu and Ye in \cite{YL}, more precisely,
%a dynamical system $(X,T)$ is called \emph{mean equicontinuous} if for every $\ep>0$,
%there exists a $\delta>0$ such that whenever $x,y\in X$ with $d(x,y)<\delta$,
%\[\limsup_{n\to\infty}\frac{1}{n}\sum_{i=0}^{n-1}d(T^ix,T^iy)<\ep.\]
%They show that a dynamical system is mean equicontinuous if and only if it is mean-L-stable.
In the study of a dynamical system with bounded complexity (defined by using the mean metrics), recently
Huang, Li, Thouvenot, Xu and Ye \cite{HL} introduced a notion called \emph{equicontinuity in the mean}.
%Following Li, Tu and Ye, we introduce a notion called \emph{equicontinuous in the mean}.
We say that a dynamical system $(X,T)$ is \emph{equicontinuous in the mean}
if for every $\ep>0$, there exists a $\delta>0$ such that
$\frac{1}{n}\sum_{i=0}^{n-1}d(T^ix,T^iy)<\ep$ for all $n\in \Z$ and
all $x,y\in X$ with $d(x,y)<\delta$.
It was proved in \cite{HL} that for a minimal system the notions of mean equicontinuity and
equicontinuity in the mean are equivalent.
In this note we will show that a dynamical system is mean equicontinuous
if and only if it is equicontinuous in the mean (Theorem \ref{general}).

In \cite{YL} the notion of Banach (or Weyl) mean equicontinuity was introduced, and the authors asked if for a minimal system
Banach mean equicontinuity is equal to mean equicontinuity. This question was answered positively in \cite{DG16}.
In this note we will show that in fact for any system the two notions are equivalent (Theorem \ref{thm:Weyl-mean-equ}). Moreover,
in \cite{YL} the authors showed that if $(X,T)$ is mean equicontinuous, then its regionally proximal relation is equal to the Banach proximal relation.
In this note we will prove that the converse statement is also valid (Theorem \ref{thm:transitive-mean-eq}).

Moreover, we define a notion called
\emph{regionally proximal relation in the mean}
and we show that the mean equicontinuous structure relation
is the smallest closed invariant equivalence relation that contains regionally proximal relation in the mean (Theorem \ref{MAX ME}).

\medskip

The note is organized as follows. In Section 2, the basic notions used in the note are introduced. In Section 3,
among other things we show that mean equicontinuity is equal to equicontinuity in the mean.
In Section 4, we prove that if the regionally proximal relation is equal to Banach proximal relation then the system is mean equicontinuous. In Section 5, we prove the equivalence of mean equicontinuity and Weyl mean equicontinuity.
In the final section, we discuss the question which relation induces the maximal mean equicontinuous factor.

\section{Preliminaries}
In this section we recall some notions and aspects of the theory of topological dynamical systems.
\subsection{Subsets of non-negative integers}
Let $\Z$ ($\N$, $\mZ$, respectively) be the set of all non-negative integers (positive integers, integers, respectively).

Let $F$ be a subset of $\mathbb{Z}_+$ ($\N$, $\mZ$, respectively). Denote by $\#(F)$ the number of elements of $F$.

We say that $F$  has \emph{density} $D(F)$ if the \emph{lower density} of $F$ ($\underline{D}(F)$) is equal to the \emph{upper density} of $F$ ($\overline{D}(F)$),
that is, $D(F)=\overline{D}(F)=\underline{D}(F)$,
where
\[ \underline{D}(F)=\liminf_{n\to\infty} \frac{\#(F\cap[0,n-1])}{n}\]
and
\[ \overline{D}(F)=\limsup_{n\to\infty} \frac{\#(F\cap[0,n-1])}{n}.\]

Similarly,
we say that $F$ has \emph{Banach density} if the \emph{lower Banach density} of $F$ ($BD_*(F)$) is equal to the \emph{upper Banach density} of $F$ ($BD^*(F)$),
that is, $BD(F)=BD_{*}(F)=BD^{*}(F)$,
where,
$$
BD_{*}(F)=\liminf_{N-M \to \infty} \frac{\#(F\cap [M,N])}{N-M+1}
$$
and
$$
BD^{*}(F)=\limsup_{N-M \to \infty} \frac{\#(F\cap [M,N])}{N-M+1}.
$$

\subsection{Compact metric spaces}
Denote by $(X,d)$ a compact metric space.
For $x\in X$ and $\ep>0$, denote $B(x,\ep)=\{y\in X\colon d(x,y)<\ep\}$.
We denote by $\diam(X)$ the diameter of $X$ given by $\diam(X)=\sup_{x, y\in X}d(x,y)$,
the product space $X\times X=\{(x,y)\colon x,y\in X\}$ and the diagonal $\Delta_X=\{(x,x)\colon x\in X\}$.
%A subset of $X$ is called a \emph{$G_\delta$ set}
%if it can be expressed as a countable intersection of open sets.

Let $C(X)$ be the set of continuous complex value functions on $X$  with the supremum norm
$\Vert f\Vert=\sup_{x\in X}|f(x)|$. We denote by $C(X)^*$ the dual space of $C(X)$.

\subsection{Topological dynamics}

Let $(X,T)$ be a dynamical system.
A non-empty closed
invariant subset $Y \subset X$ (i.e., $TY \subset Y$ ) defines naturally a subsystem $(Y,T)$ of $(X,T)$.
A system $(X,T)$ is called minimal if it contains no proper subsystem.
Each point belonging
to some minimal subsystem of $(X,T)$ is called a minimal point.
The orbit of a point $x\in X$ is the set $Orb(x,T)=\{x,Tx,T^2x, \ldots,\}$.
The set of limit points
of the orbit $Orb(x,T)$ is called the $\omega$-limit set of $x$, and is denoted by $\omega(x,T)$.
For $x \in X$ and $U,V \subset X$, put $N(x,U) = \{n \in\Z: T^nx \in U\}$ and
$N(U,V) = \{n\in \Z: U\cap T^{-n}V\neq \emptyset\}$. Recall that a dynamical system $(X,T)$ is called topologically transitive
(or just transitive) if for every two non-empty open subsets $U,V$ of $X$ the set $N(U,V)$
is infinite. Any point with dense orbit is called a transitive point. Denote the set of all
transitive points by $Trans(X,T)$. It is well known that for a transitive system, $Trans(X,T)$
is a dense $G_\delta$ subset of $X$.
%Let $(X,T)$ be a dynamical system.
%We denote by $Orb(x,T)=\{x,Tx,T^2x, \ldots,\}$ the \emph{orbit} of $x\in X$.
%The system $(X,T)$ is called \emph{transitive} if $\overline{Orb(x,T)}=X$ for some $x\in X$, and
%such a point $x$ is called a \emph{transitive point}. Denote by $Trans(X,T)$ the set of transitive points of $(X,T)$.
%It is well known that if $(X, T)$ is transitive, then $Trans(X, T)$ is a dense $G_\delta$ subset of $X$.

For two dynamical systems $(X,T)$ and $(Y,S)$. Let $\pi\colon X\to Y$ be a continuous map.
If $\pi$ is surjective with $\pi\circ T=S\circ \pi$, then we say that $\pi$ is a \emph{factor map},
the system $(Y, T)$ is a \emph{factor} of $(X, T)$ or $(X, T)$ is an \emph{extension} of $(Y, T)$.
If $\pi$ is a homeomorphism, then we say that $\pi$ is a \emph{conjugacy} and
that the dynamical systems $(X,T)$ and $(Y,S)$ are \emph{conjugate}.
By the Halmos and von Neumann theorem (see \cite[Theorem 5.18]{W82}),
a minimal system  is equicontinuous if and only if it is
conjugate to a minimal rotation on a compact abelian metric group.

A pair $(x,y)\in X\times X$ is said to be \emph{proximal}
if for any $\ep>0$, there exists a positive integer $n$ such that $d(T^nx,T^ny)<\ep$.
Let $P(X,T)$ denote the collection of all proximal pairs in $(X,T)$.
If any pair of two points in $X$ is proximal, then we say that the dynamical system $(X,T)$ is \emph{proximal}.

A pair $(x,y)\in X\times X$ is said to be \emph{Banach proximal}
if for any $\ep>0$, $d(T^nx,T^ny)<\ep$ for all $n\in \Z$ except a set of zero Banach density.
Let $BP(X,T)$ denote the collection of all Banach proximal pairs in $(X,T)$.
See \cite{LT14} for a detailed study on Banach proximality.

A pair $(x,y)$ is called \emph{regionally proximal}
if for every $\ep>0$, there exist two points  $x',y'\in X$
with $d(x,x')<\ep$ and $d(y,y')<\ep$, and a positive integer $n$ such that $d(T^nx',T^ny')<\ep$.
Let $Q(X,T)$ be the set of all regionally proximal pairs in $(X,T)$.
Clearly, $Q(X,T)\supset P(X,T)\supset BP(X,T)$.

%If $\pi\colon (X,T)\to (Y,S)$ is a factor map, then $R_\pi=\{(x,x')\in X\times X\colon \pi(x)=\pi(x')\}$
%is closed $T\times T$-invariant equivalence relation, that is $R_\pi$ is a closed subset of $X \times X$ and
%if $(x,x')\in R_\pi$, then $(Tx,Tx')\in R_\pi$.
%Conversely, if $R$ is a closed $T\times T$-invariant equivalence relation on $X$,
%then the quotient space $X/R$ is a compact metric space
%and $T$ naturally induces an action on $X/R$ by $T_R([x])=[Tx]$.
%Then $(X/R,T_R)$ forms a dynamical system and the quotient map $\pi_R\colon X\to X/R$ is a factor map.
%Hence there is a one-to-one correspondence between factors and closed invariant equivalence relations,
%we will use them interchangeably.
A factor map  $\pi\colon (X,T)\to (Y,S)$ is called \emph{proximal} (\emph{Banach proximal}, respectively)
 if whenever $\pi(x)=\pi(y)$ the pair $(x,y)$ is proximal (Banach proximal, respectively ).

The factor $\pi\colon(X,T)\to (Y,S)$ is the maximal equicontinuous factor if the system $(Y,T)$
is equicontinuous and for any other factor map $\phi\colon(X,T)\to (Z,U)$,
where $(Z,U)$ is equicontinuous, there exists a factor map $\psi\colon(Y,S)\to (Z,U)$
such that $\phi=\psi \circ \pi$.
It is thus unique up to conjugacy and therefore referred to as the \emph{maximal equicontinuous factor}.
Let $\pi\colon(X,T)\to (Y,S)$ be the factor map to the maximal equicontinuous factor.
The equivalence relation $R_\pi=\{(x,y)\in X\times X: \pi(x)=\pi(y) \}$ is called the \emph{equicontinuous structure relation}.
It is shown in~\cite{EG60} that the equicontinuous structure relation is the smallest closed invariant equivalence relation
containing the regionally proximal relation.

\subsection{Invariant measures}
For a dynamical system  $(X,T)$, we denote by $M(X,T)$ the set of $T$-invariant regular Borel probability measures on $X$.
It is well known that $M(X, T)$ is always nonempty.
We say that $(X,T)$ is \emph{uniquely ergodic} if $M(X,T)$ consists a single measure.
%any dynamical system $(X,T)$ admits at least one
%$T$-invariant regular Borel probability measures.
We regard $M(X)$ as a closed convex subset of $C(X)^*$,
equipped with the weak$^*$ topology. Then $M(X)$ is a compact metric space.
An invariant measure is ergodic if and only if it is an extreme point of $M(X,T)$.

Let $\mu \in M(X, T)$. We define the \emph{support} of $\mu$ by
$\supp(\mu)$=\{$x\in X$: $\mu(U)>0$ for any neighborhood $U$ of $x\}$.
The \emph{support} of a dynamical system $(X,T)$, denoted by $\supp(X,T)$,
is the smallest closed subset $C$ of $X$ such that $\mu(C)=1$ for all $\mu\in M(X,T)$.

The action of $T$ on $X$ induces an action on
$M(X)$ in the following way: for $\mu\in M(X)$ we define $T\mu$ by
\[\int_X f(x)\ \textrm{d}T\mu(x)= \int_X f(Tx)\textrm{d}\mu(x), \quad \forall f\in C(X).\]
Hence $(M(X),T)$ is also a topological dynamical system.

%Let $(X,T)$ be a dynamical system.

%If $(X,T)$ is a minimal rotation on a compact abelian metric group, then it is uniquely ergodic and
%the Haar measure is the only invariant measure.

Fix a measure space $(X, \mathcal{B}, \mu)$. If $f$ and $g$ are functions on $X$,
we denote by $f\otimes g$ the function on $X\times X$ given by $f\otimes g(x,x')=f(x)g(x')$
and by
$L^{\infty}(X)\otimes L^{\infty}(X) $
we denote the algebra of functions on $X\times X$ that are finite sums of
functions $f\otimes g$ with $f, g \in L^{\infty}(X, \mathcal{B}, \mu)$.
We denote by $\mu_{\Delta}$ the diagonal measure on $X\times X$ given by $\int f(x,x')\textrm{d}\mu_{\Delta}(x,x')=\int f(x,x)\textrm{d}\mu(x)$.
We notice that $\mu_\Delta (A\times B)=\mu(A\cap B)$ for any $A, B \in \mathcal{B}$.

For a dynamical system $(X,T)$, $f\in C(X)$ and $n\in\N$, let
$f_n(x)=\frac{1}{n}\sum_{i=0}^{n-1} f(T^ix).$ The following theorem is well known.

\begin{thm}\cite{O52} \label{thm:unique-ergodic}
Let $(X,T)$ be a dynamical system. Then the following conditions are equivalent:
\begin{enumerate}
  \item $(X,T)$ is uniquely ergodic;
  \item for each $f\in C(X)$, $\{f_n\}_{n=1}^\infty$ converges uniformly on $X$ to a constant;
\item for each $f\in C(X)$, there is a subsequence $\{f_{n_k}\}_{k=1}^\infty$
which converges pointwise on $X$ to a constant.
\item $(X,T)$ contains only one minimal set, and
for each $f\in C(X)$, $\{f_n\}_{n=1}^\infty$ converges uniformly on $X$.
\end{enumerate}
\end{thm}

\section{Mean equicontinuity and equicontinuity in the mean}
In this section we will show that mean equicontinuity is equal to equicontinuity in the mean.
Moreover, we will discuss what kinds of minimal sets can appear in a
transitive mean equicontinuous system.
\subsection{Mean equicontinuity and equicontinuity in the mean}
We start with the following characterizations of equicontinuous in the mean systems.
To do this, we need a simple lemma.

\begin{lem}\cite[Lemma 3.2]{YL} \label{same}
Let $(X,T)$ and $(Y,S)$ be two dynamical systems.
Then $(X\times Y,T\times S)$ is mean equicontinuous if and only if
both $(X,T)$ and $(Y,S)$ are mean equicontinuous.
\end{lem}

\begin{thm}\label{thm:unifon-conv-equi}
Let $(X,T)$ be a dynamical system. Then the following conditions are equivalent:
\begin{enumerate}
  \item $(X,T)$ is equicontinuous in the mean;
  \item for each $f\in C(X\times X)$,
  the sequence $\{\frac{1}{n}\sum_{i=0}^{n-1} f\circ(T^i\times T^i)\}_{n=1}^\infty$
  is uniformly equicontinuous;
  \item for each $f\in C(X\times X)$, the sequence
  $\{\frac{1}{n}\sum_{i=0}^{n-1} f\circ(T^i\times T^i)\}_{n=1}^\infty$ is uniformly convergent to a
  $T\times T$-invariant continuous  function $f^*\in C(X\times X)$.
\end{enumerate}
\end{thm}
\begin{proof} We only present the proof (1) implies (2) and the rest
is similar to the proof of \cite[Theorem 3.3]{YL}.

To make the idea of the proof clearer, when proving (1)$\Rightarrow$(2),
we assume $f\in C(X)$ instead of $f\in C(X\times X)$, because if $(X,T)$ is equicontinuous in the mean
if and only if so is $(X\times X,T\times T)$.

(1)$\Rightarrow$(2)
Fix $f\in C(X)$ and $\ep>0$.
By continuity of $f$,
there exists $\eta\in(0,\frac{\ep}{2\Vert f\Vert})$ such that if $x$, $y\in X$ with $d(x,y)<\eta$
then $|f(x)-f(y)|<\frac{\ep}{2}$.
As $(X,T)$ is equicontinuous in the mean,
there is $\delta\in(0,\eta)$ such that if $x,y\in X$ with $d(x,y)<\delta$, one has
\[\frac{1}{n}\sum_{i=0}^{n-1}d(T^ix,T^iy)<\eta^2,\ n=1,2,\dotsc.\]
For every $n\in \N$ and $x,y\in X$ with $d(x,y)<\delta$, let
\[I_n(x,y)=\{i\in [0,n-1]\colon d(T^ix,T^iy)\geq \eta\}.\]
Then $\#(I_n(x,y))\leq \eta n$.
So for every $n\in \N$ and $x,y\in X$ with $d(x,y)<\delta$, we have
\begin{align*}
\Bigl\vert\frac{1}{n}\sum_{i=0}^{n-1} f(T^ix)-\frac{1}{n}\sum_{i=0}^{n-1} f(T^iy)\Bigr\vert
&\leq \frac{1}{n} \sum_{i=0}^{n-1}\bigl\vert f(T^ix)-f(T^iy)\bigr\vert\\
&\leq \frac{1}{n}\Bigl(\sum_{i\in I_n(x,y)} 2 \Vert f\Vert+
 \sum_{i\in [0,n-1]\setminus I_n(x,y)}\bigl\vert f(T^ix)-f(T^iy)\bigr\vert\Big)\\
 &\leq 2\eta \Vert f\Vert +\frac{\ep}{2}<\ep.
\end{align*}
This shows that $\{\frac{1}{n}\sum_{i=0}^{n-1} f\circ T^i\}_{n=1}^\infty$
is uniformly equicontinuous.
\end{proof}

Before proving the main result of this section we give a proof of a result in \cite{O52} which is outlined there.
We need the following lemmas.

\begin{lem}\label{mean-eq-minimal}\cite{HL}
Let $(X,T)$ be  a minimal  dynamical system.
Then $(X,T)$ is mean equicontinuous
if and only if it is equicontinuous in the mean.
\end{lem}

\begin{lem}\label{lem:mean-eq-BP}\cite[Theorem 3.5]{YL}
Let $(X,T)$  be a dynamical system. If $(X,T)$ is mean equicontinuous, then $Q(X,T)=P(X,T)=BP(X,T)$ and
it is a closed invariant equivalence relation.
\end{lem}

\begin{lem} \cite{LT14} \label{PRBD}
Let $(X,T)$ be a dynamical system. Then the support of $(X,T)$ is the smallest closed subset $K$
of $X$ such that for every $x\in X$ and every open neighborhood $U$ of $K$, $N(x,U)$ has Banach density one.
\end{lem}

Now we are ready to show%The following theorem is shown in \cite{O52}, and we give a direct proof.
\begin{thm}\label{transitive}
Let $(X,T)$ be a dynamical system.
If $(X,T)$ is mean equicontinuous,
then for every $x\in X$, $(\overline{\text{Orb}(x,T)},T)$ is uniquely ergodic.
In particular,
if $(X,T)$ is also transitive, then $(X,T)$ is uniquely ergodic.
\end{thm}
\begin{proof}
Without loss of generality, we can assume that $X=\overline{Orb(x,T)}$.

If  $M_1,M_2$ are two minimal sets in $(X,T)$.
By Auslanser-Ellis theorem, there exist $y_1\in M_1$ and $y_2\in M_2$
such that $(x,y_1)$ and  $(x,y_2)$ are both proximal.
For a given $\ep>0$,set
$$
A_1=\{n\in\Z: d(T^nx,T^ny_1)<\ep/2\}\ \text{and}\ A_2=\{n\in\Z: d(T^nx,T^ny_2)<\ep/2\}.
$$
By Lemma~\ref{lem:mean-eq-BP}, $A_1\cap A_2\not=\emptyset$ which implies that $(y_1,y_2)$ is proximal.
As $y_1$ and $y_2$ are minimal points, then their orbit closures
are equal which deduces that $M_1=M_2$.
So there is only one minimal set in $(X,T)$, denoted by $M$.

It is clear that $(M,T)$ is also mean equicontinuous.
By Lemma \ref{mean-eq-minimal}, $(M,T)$ is equicontinuous in the mean.
Then $(M,T)$ is  uniquely ergodic by Theorem  \ref{thm:unique-ergodic} and Theorem \ref{thm:unifon-conv-equi}.
For every $z\in X$, by Auslanser-Ellis theorem again,
there  exists a point $y\in M$ such that $(z,y)$ is proximal.
By Lemma~\ref{lem:mean-eq-BP}, $(z,y)$ is a Banach proximal piont.
So for every open neighborhood $U$ of $M$ and any $z\in X$,
$N(z,U)$ has Banach density one. Then by Lemma \ref{PRBD} we have $\supp(X,T)\subset M$.
As $(M,T)$ is  uniquely ergodic, so is $(X,T)$.
\end{proof}

Now we begin to prove the main result of this section. We need the following lemma.
%Before showing the equivalence between mean equicontinuity and equicontinuity in the mean, we need some lemmas.

\begin{lem}\label{minimal-point}
Let $(X,T)$ be mean equicontinuous system and $\nu$ be an ergodic measure on $X$,
then every point of $\text{supp} (\nu)$ is minimal.
\end{lem}

\begin{proof}
$(\text{supp}(\nu),T)$ is a transitive system since $\nu$ is an ergodic measure on $X$.
By Theorem \ref{transitive},
$(\text{supp}(\nu),T)$ is uniquely ergodic,
and so,
it is minimal.
\end{proof}

Now we are ready to show the main result. Note that our method is different from
the proof for the minimal case.

\begin{thm}\label{general}
$(X,T)$ is mean equicontinuous if and only if equicontinuous in the mean.
\end{thm}

\begin{proof}
If $(X,T)$ is equicontinuous in the mean,
it is clear that $(X,T)$ is mean equicontinuous.

Now assume that $(X,T)$ is mean equicontinuous.
If $(X,T)$ is not equicontinuous in the mean, then there are
$x_k,y_k,z\in X,n_k\in \Z,k=1,2,\cdots$ and $\ep_0>0$
such that $\lim_{k\rightarrow\infty }x_k=z=\lim_{k\rightarrow\infty }y_k$
and for every $k\in \Z$
$$
\frac{1}{n_k}\sum_{i=0}^{n_k-1}d(T^ix_k,T^iy_k)\geq \ep_0.
$$

Let $\mu_k=\frac{1}{n_k}\sum_{i=0}^{n_k-1}\delta_{(T^ix_k,T^iy_k)}$, then $\mu_k\in M(X\times X)$,
we may assume $\mu_k\rightarrow \mu$(otherwise we may consider the subsequence),
where $\mu\in M(X\times X,T\times T)$.

We claim that $\mu(\text{supp}(\mu )\setminus \Delta_X)>0$.
Actually, $d(\cdot,\cdot)$ is a continuous function on $X\times X$, then we have
$$
\int_{X\times X}d(x,y)\textrm{d}\mu_k\longrightarrow \int_{X\times X}d(x,y)\textrm{d}\mu
$$
and
$$
\int_{X\times X}d(x,y)\textrm{d}\mu_k=\frac{1}{n_k}\sum_{i=0}^{n_k-1}d(T^ix_k,T^iy_k)\geq \ep_0
$$
which implies
$$\mu(\text{supp}(\mu)\setminus \Delta_X)>0.$$
By the ergodic decomposition, we have $\nu(\text{supp}(\mu)\setminus \Delta_X)>0$
for some ergodic measure on $X\times X $. Thus, there exists a minimal point in $\text{supp}(\mu)\setminus \Delta_X$,
since $\text{supp}(\nu)$ is a minimal set by Lemma \ref{minimal-point}. Denote this minimal point by $(u,v)$.

For $l\in\N$, let $B_l=\{(x,y)\in X\times X:d((x,y),(u,v))<\frac{1}{l} \}$,
then
$$
\mu(B_l)>0  \ \text{and}\
\mu_k(B_l)=\frac{1}{n_k}\#(\{0\leq i \leq n_k:(T^ix_k,T^iy_k)\in B_l\}).
$$
There are infinitely many $k\in \Z$ with $0\leq m_k \leq n_k$ such that $(T^{m_k}x_k,T^{m_k}y_k)\in B_l$,
since $0<\mu(B_l)\leq \liminf_{k\rightarrow\infty}\mu_k(B_l)$ for every $l\in \N$.

Put $$\delta=d(\overline{Orb((z,z),T\times T)},\overline{Orb((u,v),T\times T)}).$$ Then, $\delta>0$, since $\overline{Orb((u,v),T\times T)}\cap \Delta_X=\emptyset$.
As $(X,T)$ is mean equicontinuous, so is $(X\times X,T\times T)$ by Lemma \ref{same}. Then,
for $\frac{1}{4}\delta$, there is $\eta>0$ such that if $d((x,y),(x',y'))<\eta$,
then \[\limsup_{n\to\infty}\frac{1}{n}\sum_{i=0}^{n-1}d((T\times T)^i(x,y),(T\times T)^i(x',y'))<\frac{\delta}{4}.\]

We can choose $k\in \mathbb{Z_{+}}$ with
 $d((x_k,y_k),(z,z))<\eta$ and $d((T^{m_k}x_k,T^{m_k}y_k),(u,v))<\eta$,
then
 \[\limsup_{n\to\infty}\frac{1}{n}\sum_{i=0}^{n-1}d((T\times T)^i(x_k,y_k),(T\times T)^i(z,z))<\frac{\delta}{4}\]
 and
  \[\limsup_{n\to\infty}\frac{1}{n}\sum_{i=0}^{n-1}d((T\times T)^i(T^{m_k}x_{k},T^{m_k}y_{k}),(T\times T)^i(u,v))<\frac{\delta}{4}\]
which implies
\[\limsup_{n\to\infty}\frac{1}{n}\sum_{i=0}^{n-1}d((T\times T)^i(T^{m_k}z,T^{m_k}z),(T\times T)^i(u,v))<\frac{\delta}{2}.\]
It is a contradiction, thus $(X,T)$ is equicontinuous in the mean.
\end{proof}

\subsection{Minimal sets in a transitive mean equicontinuous  system}
In Theorem \ref{transitive} we have shown that a transitive mean equicontinuous system is uniquely ergodic,
and thus it contains a unique minimal subset.
Here we will discuss what kinds of minimal sets can appear in a
transitive mean equicontinuous system.

\begin{thm} We have the following observations.
 \begin{enumerate}
 \item If $(X,T)$ is weakly mixing and mean equicontinuous, then the unique minimal set is a fixed point.
 \item If $(X,T)$ is totally transitive and mean equicontinuous, then the unique minimal set is totally minimal and mean equicontinuous.
 Moreover, any totally transitive, minimal, mean equicontinuous system can be realized in a totally transitive non-minimal
 mean equicontinuous system.
 \item If $(X,T)$ is transitive and mean equicontinuous, then the unique minimal set is mean equicontinuous.
 Moreover, any minimal, mean equicontinuous system can be realized in a transitive non-minimal
 mean equicontinuous system.
\end{enumerate}
\end{thm}
\begin{proof}
(1). If $(X,T)$ is mean equicontinuous and weakly mixing,
 then $(X\times X,T\times T)$ is transitive, thus it is uniquely ergodic by Theorem \ref{transitive}.
 As we know $\mu\times \mu$ and $\mu_\Delta$ are invariant measure on $X\times X$
 for any invariant measure $\mu$ on $X$,
 thus $\mu\times \mu=\mu_\Delta$.
 $\mu_\Delta(\Delta_X)=1$ implies $\mu$ must be $\delta_x$ for some $x\in X$.
 Hence $(X,T)$ is also uniquely ergodic and the unique fixed point is $x$.

 \medskip
 (2). If $(X,T)$ is mean equicontinuous and totally transitive,
 it has only one minimal set by Theorem \ref{transitive}, denoted by $M$.
 Then $(M,T)$ is totally minimal.
 Actually, $(X,T^n)$ is also a transitive mean equicontinuous system for every $n\in\mathbb{N}$,
 again by Theorem \ref{transitive}, there is only one $T^n$-invariant measure
 on $X$ denoted by $\mu_n$.
Let $\mu$ be the unique invariant measure on $(X,T)$ and it is also invariant on $T^n$,
hence $\mu=\mu_n$ and $M=\text{supp}(\mu)=\text{supp}(\mu_n)$,
which implies $(M,T^n)$ is minimal. It is clear that $(X,T)$ is mean equicontinuous.

Now let $(X_1,T_1)$ be a totally minimal mean equicontinuous system and $(X_2,T_2)$ be a weakly mixing
system such that the uniquely ergodic measure is supported on a fixed point $p$. Then, $(X_2,T_2)$
is mean equicontinuous and $(X_1\times X_2, T_1\times T_2)$ is the system we want.

 \medskip
(3). The first statement follows again by Theorem \ref{transitive}. Let $(X_1,T_1)$ be a minimal mean equicontinuous system
and $(X_2,T_2)$ be a weakly mixing
system such that the uniquely ergodic measure is supported on a fixed point $p$. Then, $(X_2,T_2)$
is mean equicontinuous and $(X_1\times X_2, T_1\times T_2)$ is the system we want.
\end{proof}

\section{Regionally proximal and Banach proximal relations}
Lemma \ref{lem:mean-eq-BP} shows that for a mean equicontinuous system $(X,T)$,
we have $BP(X,T)=P(X,T)=Q(X,T)$.
%So is for system with equicontinuity in the mean by Theorem \ref{general}.
We will show the converse is also valid, i.e.
for a dynamical system $(X,T)$, if $BP(X,T)=P(X,T)=Q(X,T)$
then it is equicontinuous in the mean. In fact we will prove more by providing a series of equivalent statements, see Theorem \ref{thm:transitive-mean-eq} for details.
%In fact, we will give another characterization of mean equicontinuity and
%obtain the result as an application.
%\subsection{System $(X,T)$ with $BP(X,T)=P(X,T)=Q(X,T)$}

We start with some preparations.
The following lemma is just a simple observation.

\begin{lem}
Let $(X,T)$ be a dynamical system, if $(x,y)\in BP(X,T)$, then for any neighborhood $U$ of $\Delta_{X}$,
we have $BD(N((x,y),U))=1$.

\end{lem}

\begin{lem}\label{measure}
Let $(X,T)$ be a dynamical system.
If there exists $\mu\in M(X\times X, T\times T)$ such that $\mu(BP(X,T))=1$,
then $\mu(\Delta_X)=1$.
\end{lem}
\begin{proof}
Assume that $\mu(\Delta_{X})<1$, i.e. $\mu(BP(X,T)\setminus \Delta_{X})>0$,
As $X$ is a compact metric space there exists a closed set $F\subset \text{supp}(\mu)\cap BP(X,T)\setminus \Delta_{X}$
with $\mu(F)>0$.
By the ergodic decomposition theorem, there exists an ergodic measure $\nu$ with $\nu(F)>0$.
By Birkhoff ergodic theorem there exists $z\in F$ such that
$$ \frac{1}{n} \#(N(z,F)\cap[0,n-1])=\frac{1}{n}\sum_{i=0}^{n-1} \chi_{F}(T^{i}z)\rightarrow \int \chi_F \textrm{d}\nu=\nu(F)>0,$$
then we have $D(N(z,F))>0$.
We choose neighborhoods $U$ and $V$ of $F$ and $\Delta_X$ respectively with $U\cap V=\emptyset$,
then $\underline{BD}(N(z,U))\geq D(N(z,F))>0$.
On the other hand, we have $BD(N(z,V))=1$, since $z\in BP(X,T)$.
The contradiction shows the lemma.
\end{proof}

%The following lemma is a basic fact in general topology.

%\begin{lem}\label{topology}
%Let $(X,\tau_X)$ be a compact space and $\pi:X\rightarrow Y$ be a quotient map.
%We denote by $\tau_Y$ the topology on $Y$ induced by $\pi$.
%If there exists another topology $\tau'$ on $Y$
%such that $\pi:(X,\tau_X)\rightarrow (Y,\tau')$
%is continuous, then we have $\tau'\subset \tau_Y$.
%\end{lem}

For a minimal system the following result was known, see \cite{DG16} and \cite{YL}. We now show
it holds for a general system.

\begin{thm}\label{thm:transitive-mean-eq}
Let $(X,T)$ be a dynamical system.
Let $(Z,S)$ be the maximal equicontinuous factor of $(X,T)$
and  $\pi\colon (X,T)\to (Z,S)$ be the factor map.
Then the following conditions are equivalent:
\begin{enumerate}
	\item $(X,T)$ is mean equicontinuous;
	\item  $\pi$ is Banach proximal;
\item $BP(X,T)=P(X,T)=Q(X,T)$;
	\item $\pi\colon (X,\mu,T) \to (Z,\nu,S)$ is measure-theoretically isomorphic, where $\mu$ and $\nu$ are
	any invariant measures on $X$ and $Z$  respectively with $\pi(\mu)=\nu$;
	\item $(X,T)$ is equicontinuous in the mean.
\end{enumerate}
\end{thm}
\begin{proof}
	(1) $\Rightarrow$ (2) by Lemma \ref{lem:mean-eq-BP}. (1) $\Leftrightarrow$ (5) by Theorem \ref{general}. Moreover, it is clear
that (2) $\Leftrightarrow$ (3).
	
	(3) $\Rightarrow$ (4)
This is essentially proved in \cite[Theorem 3.8]{YL}.
Here we provide a proof for completeness.
	
Let $\mu$ be an invariant measure on $(X,T)$ and $\nu$ be the invariant measure on $(Z,S)$ with $\pi(\mu)=\nu$.
We consider the disintegration of $\mu$ over $\nu$.
That is, for a.e. $y\in Z$ we have a measure $\mu_y$ on $X$ such that $\supp(\mu_y)\subset \pi^{-1}(y)$
and $\mu=\int_{y\in Z}\mu_y \textrm{d}\nu$.
Let $W=\{(u,v)\in X\times X: \pi(u)=\pi(v)\}$.
As $\supp(\mu_y)\subset \pi^{-1}(y)$,
we have $\supp(\mu_y\times\mu_y)\subset \pi^{-1}(y)\times \pi^{-1}(y)\subset W$, a.e. $y\in Y$.
Let $\mu\times_Z\mu=\int_{y\in Z}\mu_y\times\mu_y\textrm{d}\nu$.
Then $\mu\times_Z\mu$ is an invariant measure on $(X\times X,T\times T)$.
Moreover,
\[\mu\times_Z\mu(W)=\int_{y\in Z} \mu_y\times\mu_y(W)\textrm{d}\nu=1,\]
then $\supp(\mu\times_Z\mu)\subset W$.
By Lemma \ref{measure}
we have $\mu\times_Z\mu(\Delta_X)=1$.
Since
\[\mu\times_Z\mu(\Delta_X)=\int_{y\in Y}\mu_y\times \mu_y(\Delta_X)\textrm{d}\nu=1,\]
we have $\mu_y\times\mu_y(\Delta_X)=1$ a.e. $y\in Z$.
Then for a.e. $y\in Z$,
there exists a point $c_y\in \pi^{-1}(y)$ such that $\mu_y=\delta_{c_y}$.
	
Let $Y_0$ be the collection of $y\in Z$ such that $\mu_y$ is not equal to $\delta_x$ for
any $x\in X$. Then $\nu(\cup_{i\in \Z}S^{-i}Y_0)=0$. Let $Z_0=Z\setminus \cup_{i\in \Z}S^{-i}Y_0$
and $X_0=\{c_y:y\in Z_0\}$.
Then $\nu(Z_0)=1$. Now we show that $X_0$ is a measurable set. In fact,
the map $y\mapsto\mu_y$ from $Z_0$ to $M(X)$ is measurable and
$x\mapsto \delta_x$ is an embedding. Since $Z_0$ is a measurable set and
maps are 1-1, it follows from Souslin's theorem that $X_0$ is a
measurable set, and it is clear that $\mu(X_0)=\mu(\pi^{-1}Z_0)=\nu(Z_0)=1$. By the same
argument, $\pi: (X_0,\mu,T)\to (Z_0,\nu,S)$ is a
measure-theoretic isomorphism.
	
(4) $\Rightarrow$ (5)
If $(X,T)$ is not equicontinuous in the mean, then there are
$x_k,x\in X,n_k\in \Z,k=1,2,\cdots$ and $\ep_0>0$
such that $\lim_{k\rightarrow\infty }x_k=x$
and for every $k\in \Z$
$$
\frac{1}{n_k}\sum_{i=0}^{n_k-1}d(T^ix_k,T^ix)\geq \ep_0.
$$

Let $\pi(x_k)=z_k$ and $\pi(x)=z$. We define
$$\mu_k=\frac{1}{n_k}\sum_{i=0}^{n_k-1}\delta_{(T^ix_k,T^ix)}$$
and
$$\nu_k=\frac{1}{n_k}\sum_{i=0}^{n_k-1}\delta_{(S^iz_k,S^iz)}$$
then
$$(\pi \times \pi)(\mu_k)=\nu_k.$$
By taking the subsequence,
there exists $\mu$ and $\nu$ on $M(X\times X,T\times T)$ and $M(Z\times Z,S\times S)$
respectively with $\lim_{k\rightarrow\infty}\mu_k=\mu,\lim_{k\rightarrow\infty}\nu_k=\nu$
and $(\pi \times \pi)(\mu)=\nu$.

We claim that $\mu(\text{supp}(\mu )\setminus \Delta_X)>0$.
Actually, $d(\cdot,\cdot)$ is a continuous function on $X\times X$, then we have
$$
\int_{X\times X}d(x,y)\textrm{d}\mu_k\longrightarrow \int_{X\times X}d(x,y)\textrm{d}\mu
$$
and $\int_{X\times X}d(x,y)\textrm{d}\mu_k=\frac{1}{n_k}\sum_{i=0}^{n_k-1}d(T^ix_k,T^ix)\geq \ep_0$
which implies $\mu(\text{supp}(\mu)\setminus \Delta_X)>0$.
There are open sets $U$ and $V$ of $X$ with $U\cap V=\emptyset$ and $\mu(U\times V)>0$.

Let $\mu'$ and $\nu'$ be the projection of $\mu$ and $\nu$ onto the
first component of $X$ and $Z$ respectively. It is clear that $\mu'\in M(X,T)$ and $\nu'\in M(Z,S)$.
It is easy to see $\text{supp}(\nu)\subset \overline{Orb((z,z),S\times S)}$.
Then $\nu'(\overline{Orb(z,S)})=\nu(\overline{Orb(z,S)}\times Z)=1$,
which implies $\text{supp}(\nu')\subset \overline{Orb(z,S)}$.
Furthermore,
$\nu'$ is the only invariant measure on $\overline{Orb(z,S)}$,
since $\overline{Orb(z,S)}$ is uniquely ergodic.
As for every $f,g\in C(Z)$, we have
$$\int_{Z\times Z}f(z_1)g(z_2)\textrm{d}\nu_k(z_1,z_2)\longrightarrow \int_{Z\times Z}f(z_1)g(z_2)\textrm{d}\nu(z_1,z_2)$$
and
$$\int_{Z\times Z}f(z_1)g(z_2)\textrm{d}\nu_k(z_1,z_2)=\frac{1}{n_k}\sum_{i=0}^{n_k-1}f(S^iz_{k})g(S^iz)\longrightarrow
\int_Z f(z)g(z)\textrm{d}\nu'(z),$$
thus $\nu(A\times B)=\nu'(A\cap B)$.
So $\nu'(\pi(U)\cap \pi(V))=\nu(\pi(U)\times \pi(V))\geq \mu(U\times V)>0$.

Obviously, $\pi(\mu')$ is an invariant measure on $\overline{Orb(z,S)}$,
thus we have $\pi(\mu')=\nu'$. As
$\pi\colon (X,\mu',T) \to (Z,\nu',S)$ is measure-theoretic isomorphic, we have $\nu'(\pi(U)\cap \pi(V))=\mu'(U\cap V)=0$,
it is a contradiction.
This shows $(X,T)$ is mean equicontinuous.
\end{proof}

%As an application of the above theorem we get

%\begin{thm}\label{last-theorem}
%Let $(X,T)$ be a dynamical system. If $BP(X,T)=P(X,T)=Q(X,T)$,
%then $(X,T)$ is equicontinuous in the mean.
%\end{thm}
%\begin{proof} We observe that $\pi: X\rightarrow X_{eq}$, the factor map to the maximal equicontinuous factor, is Banach proximal.
%Then we conclude the theorem by using Theorem \ref{thm:transitive-mean-eq}.
%\end{proof}

%By Lemma ~\ref{lem:mean-eq-BP}  and Theorems~\ref{last-theorem}, we have the following corollary.

%\begin{cor}\label{BPQ}
%For any dynamical system $(X,T)$,
%it is equicontinuous in the mean if and only if $BP(X,T)=P(X,T)=Q(X,T)$.
%\end{cor}

\section{Mean equicontinuity and Weyl mean equicontinuity}

Following \cite{DG16} and \cite{YL},
a dynamical system $(X,T)$ is called \emph{Banach mean equicontinuous or Weyl mean equicontinuous}
%(or \emph{Banach mean equicontinuous})
if for every $\ep>0$, there exists a $\delta>0$ such that
\[
\limsup_{n-m\to\infty}\frac{1}{n-m}\sum_{i=m}^{n-1}d(T^ix,T^iy)<\ep
\]
for all $x,y\in X$ with $d(x,y)<\delta$.
It is clear that  each Weyl mean equicontinuous system is mean equicontinuous.
It is shown in~\cite{DG16} that if a minimal system is
mean equicontinuous then it is Weyl mean equicontinuous.
In this section we show that for a general dynamical system mean equicontinuity is equivalent to
Weyl mean equicontinuity. That is, we have

\begin{thm}\label{thm:Weyl-mean-equ}
A dynamical system $(X,T)$ is mean equicontinuous if and only if
it is Weyl mean equicontinuous.
\end{thm}

Before proving the Theorem, we need the following lemma.

\begin{lem}
If a dynamical system $(X,T)$ is uniquely ergodic, then
for any $f\in C(X, \mathbb{R})$ and $x\in X$,
\[\lim_{n-m\to\infty}\frac{1}{n-m}\sum_{i=m}^{n-1}
f(T^ix)=\lim_{n\to\infty}\frac{1}{n}\sum_{i=0}^{n-1}
f(T^ix).\]
\end{lem}
\begin{proof}
Let $\mu$ be the unique invariant measure on $(X,T)$.
Then for any $f\in C(X, \mathbb{R})$ and $x\in X$,
\[ \lim_{n\to\infty}\frac{1}{n}\sum_{i=0}^{n-1}
f(T^ix)=\int f \textrm{d}\mu.\]
If the conclusion does not hold, then there exist $f\in C(X, \mathbb{R})$, $x\in X$ and two sequences $\{n_k\}$ and $\{m_k\}$ with
$n_k- m_k\to\infty $ such that
\[\lim_{n_k-m_k\to\infty}\frac{1}{n_k-m_k}\sum_{i=m_k}^{n_k-1}
f(T^ix) =c \neq \int f \textrm{d}\mu.\]
As $M(X)$ is compact, passing to a subsequence if necessary
we may assume that
\[ \lim_{k\to\infty}
\frac{1}{n_k-m_k}\sum_{i=m_k}^{n_k-1}\delta_{T^ix}
=\nu.\]
It is easy to see that $\nu$ is an invariant measure.
As $(X,T)$ is unqiuely ergodic then $\nu=\mu$.
So
\[\lim_{n_k-m_k\to\infty}\frac{1}{n_k-m_k}\sum_{i=m_k}^{n_k-1}
f(T^ix) = \int f\textrm{d}\mu.\]
This is a contradiction.
\end{proof}

\begin{proof}[Proof of Theorem~\ref{thm:Weyl-mean-equ}]
As $(X,T)$ is mean equicontinuous, then so is $(X\times X,T\times T)$.
Fix $(x,y)\in X\times X$.
By Theorem \ref{transitive}, $(\overline{Orb((x,y),T\times,T)},T\times T)$ is uniquely ergodic.
Now applying the above theorem to the distance function $d(\cdot,\cdot)$ and $(x,y)$, we get
\[
\lim_{n-m\to\infty}\frac{1}{n-m}\sum_{i=m}^{n-1}d(T^ix,T^iy)=
\lim_{n\to\infty}\frac{1}{n}\sum_{i=0}^{n-1}d(T^ix,T^iy).
\]
Then the result follows from the definition.
\end{proof}

We now give the following conclusion to end this section.

\begin{thm}
Let $(X,T)$ be a mean equicontinuous system, then
for every
$\varepsilon >0$, there are $\delta>0$ and $N>0$, such that whenever $d(x,y)<\delta$, we have
$$\frac{1}{n}\sum_{i=j}^{n+j-1} d(T^ix,T^iy)<\varepsilon$$
 for all $ n\geq N$ and $j\geq 0$.
\end{thm}

\begin{proof}
Assume that there are
$x_k,y_k,z\in X,n_k,j_k\in \Z,k=1,2,\cdots$ and $\ep_0>0$
such that $\lim_{k\rightarrow\infty }x_k=z=\lim_{k\rightarrow\infty }y_k$
and for every $k\in \Z$
$$
\frac{1}{n_k}\sum_{i=j_k}^{n_k+j_k-1}d(T^ix_k,T^iy_k)\geq \ep_0.
$$

Let $\mu_k=\frac{1}{n_k}\sum_{i=j_k}^{n_k+j_k-1}\delta_{(T^ix_k,T^iy_k)}$, and then $\mu_k\in M(X\times X)$.
We may assume $\mu_k\rightarrow \mu$(otherwise we may consider the subsequence),
where $\mu\in M(X\times X,T\times T)$.

We claim that $\mu(\text{supp}(\mu )\setminus \Delta_X)>0$.
In fact, $d(\cdot,\cdot)$ is a continuous function on $X\times X$, then we have
$$
\int_{X\times X}d(x,y)\textrm{d}\mu_k\longrightarrow \int_{X\times X}d(x,y)\textrm{d}\mu
$$
and
$$
\int_{X\times X}d(x,y)\textrm{d}\mu_k=\frac{1}{n_k}\sum_{i=j_k}^{n_k+j_k-1}d(T^ix_k,T^iy_k)\geq \ep_0
$$
which implies
$$\mu(\text{supp}(\mu)\setminus \Delta_X)>0.$$
By ergodic decomposition theorem, we have $\nu(\text{supp}(\mu)\setminus \Delta_X)>0$
for some ergodic measure $\nu$ on $X\times X $, thus there exists a minimal point $(u,v)$ in $\text{supp}(\mu)\setminus \Delta_X$
since $\text{supp}(\nu)$ is a minimal set by Lemma \ref{minimal-point}.

Let $B_l=\{(x,y)\in X\times X:d((x,y),(u,v))<\frac{1}{l} \}$.
Then we have
$$\mu(B_l)>0 \ \text{and}\
\mu_k(B_l)=\frac{1}{n_k}\#(\{j_k\leq i \leq n_k+j_k-1:(T^ix_k,T^iy_k)\in B_l\}).$$
Thus for any $l\in \Z$ there exist infinte $k\in \Z$ with $j_k\leq m_k \leq n_k+j_k-1$ such that $(T^{m_k}x_k,T^{m_k}y_k)\in B_l$,
since $0<\mu(B_l)\leq \liminf_{k\rightarrow\infty}\mu_k(B_l)$.

Put $$\delta=d(\overline{Orb((z,z),T\times T)},\overline{Orb((u,v),T\times T)})$$
then $\delta>0$, since $\overline{Orb((u,v),T\times T)}\cap \Delta_X=\emptyset$.
As $(X,T)$ is mean equicontinuous,
so is $(X\times X,T\times T)$ by Lemma \ref{same}. Thus,
for $\frac{1}{4}\delta$, there is $\eta>0$ such that if $d((x,y),(x',y'))<\eta$
then \[\limsup_{n\to\infty}\frac{1}{n}\sum_{i=0}^{n-1}d((T\times T)^i(x,y),(T\times T)^i(x',y'))<\frac{\delta}{4}.\]

There are infinitely many $k\in \Z$ with
 $d((x_k,y_k),(z,z))<\eta$ and $d((T^{m_k}x_k,T^{m_k}y_k),(u,v))<\eta$,
then
 \[\limsup_{n\to\infty}\frac{1}{n}\sum_{i=0}^{n-1}d((T\times T)^i(x_k,y_k),(T\times T)^i(z,z))<\frac{\delta}{4}\]
 and
  \[\limsup_{n\to\infty}\frac{1}{n}\sum_{i=0}^{n-1}d((T\times T)^i(T^{m_k}x_{k},T^{m_k}y_{k}),(T\times T)^i(u,v))<\frac{\delta}{4}\]
which implies
\[\limsup_{n\to\infty}\frac{1}{n}\sum_{i=0}^{n-1}d((T\times T)^i(T^{m_k}z,T^{m_k}z),(T\times T)^i(u,v))<\frac{\delta}{2}.\]
It is a contradiction which shows the theorem.
\end{proof}

\section{Mean equicontinuous relation}

It is well known that the equicontinuous structure relation
is the smallest closed invariant equivalence relation containing the regionally proximal relation.
In \cite{YL} the authors showed that each topological dynamical system admits a maximal mean equicontinuous factor.
Inspired by the above ideas, we now define a new notation called \emph{a pair sensitive in the mean}
and introduce the mean equicontinuous structure relation. We show that the maximal mean equicontinuous factor
is induced by the smallest invariant closed equivalence relation containing the relation of sensitivity in the mean.
%Moreover we showed that for any dynamical system admits a maximal mean equicontinuous factor.

\begin{defn}
Let $(X,T)$ be a dynamical system.
We say $(x,y)$ is \emph{a pair sensitive in the mean},
if $x=y$ or
for any $\tau>0$, there exists $c=c(\tau)>0$ such that
for every $\epsilon>0$,
there exist $x',y'\in X$ and $n\in \N$ such that
$d(x',y')<\ep$ and
$$ \frac{1}{n}\#(\{0\leq i\leq n-1:d(T^ix',x)<\tau,d(T^iy',y)<\tau\})>c$$
Let $Q_{me}(X,T)$ be the set of all pairs sensitive in the mean  in $(X,T)$,
and we call that  \emph{the relation of sensitivity in the mean}.
\end{defn}
Clearly, if $T$ is a homeomorphism, then $Q_{me}(X,T)\subset Q(X,T)$.
Let $S_{me}(X, T)$ be the smallest closed $T\times T$ invariant equivalence relation
such that $X/S_{me}(X, T)$ is a mean equicontinuous system.
We will show that $S_{me}(X, T)$ is the smallest closed $T\times T$ invariant equivalent relation that contains
the relation of sensitivity in the mean.
This will be done through the following lemmas.

First we observe that

\begin{lem} \label{NOT ME}
Let $(X,T)$ be a dynamical system.
Then $(X,T)$ is not mean equicontinuous system if and only if
there exists $x, x_{k}\in X, n_{k}\in \mathbb{N}$ and $\varepsilon_{0}>0$,
such that $x_{k}\rightarrow x, \frac{1}{n_{k}}\sum_{i=0}^{n_{k}-1}d(T^{i}x,T^{i}x_{k})\geq \varepsilon_{0}$.
\end{lem}

It is easy to check:

\begin{lem}
 Let $\pi:(X,T)\rightarrow (Y,S)$ be a factor map.
 If $(x,y)\in Q_{me}(X, T)$, then we have $(\pi(x),\pi(y))\in Q_{me}(Y,S)$.
\end{lem}

\begin{lem} \label{ME-Q}
 Let $(X,T)$ be a dynamical system, then $(X,T)$ is mean equicontinuous if and only if
 $Q_{me}(X, T)=\Delta_X.$
\end{lem}
\begin{proof}
If $(X,T)$ is mean equicontinuous, it is clear that $Q_{me}(X, T)=\Delta_X$.

Conversely, assume that $Q_{me}(X, T)=\Delta_X$. Suppose that $(X,T)$ is not mean equicontinuous.
By Lemma ~\ref{NOT ME}, there are $x_k,x \in X,n_k\in\Z$ and $\ep_0>0$
such that $\lim_{k\rightarrow\infty}d(x_k,x)=0$
 and for every $k\in \N$, we have
 $$
 \frac{1}{n_k}\sum_{i=0}^{n_k-1}d(T^ix_k,T^ix)\geq\ep_0.
 $$
Let $\mu_k=\frac{1}{n_k}\sum_{i=0}^{n_k-1}\delta_{(T^ix_k,T^ix)}$, then $\mu_k\in M(X\times X)$,
we may assume $\mu_k\rightarrow \mu$(otherwise we may consider the subsequence),
where $\mu\in M(X\times X,T\times T)$.
 From the proof of Theorem ~\ref{general}, it follows that $\mu(\text{supp}(\mu)\setminus \Delta_X)>0$.

Let $(y,z)\in \text{supp}(\mu)\setminus \Delta_X$.
Fix $\tau>0$, choose $l\in\N$ such that $\frac{1}{l}<\tau$.
Let $B_l=\{(u,v)\in X\times X:d((y,z),(u,v))<\frac{1}{l} \}$,
where $d((y,z),(u,v))=d(y,u)+d(z,v)$
then
$$
\mu(B_l)>0  \ \text{and}\
\mu_k(B_l)=\frac{1}{n_k}\#(\{0\leq i \leq n_k-1:d(T^ix_k,y)+d(T^ix,z)<\frac{1}{l}\}).
$$
There exist infinite $k\in \Z$ such that $\mu_k(B_l)>\frac{\mu(B_l)}{2}>0$
since $0<\mu(B_l)\leq \liminf_{k\rightarrow\infty}\mu_k(B_l)$.
For $\ep>0$, we can choose $k\in \Z$ satisfying above proposition with $d(x_k,x)<\ep$,
hence
$$
\frac{1}{n_k}\#(\{0\leq i \leq n_k-1:d(T^ix_k,y)<\tau,d(T^ix,z)<\tau\})>\frac{\mu(B_l)}{2}.
$$
It follows that $(y,z)\in Q_{me}(X, T)$.
It is a contradiction which implies the lemma.
\end{proof}

\begin{lem}\label{mean equicontinuous factor}
Let $(X,T)$ be a dynamical system and $\mathcal{A}(Q_{me}(X, T))$ be the
smallest closed $T\times T$ invariant equivalence relation containing $Q_{me}(X, T)$,
then $X/\mathcal{A}(Q_{me}(X, T))$ is the maximal mean equicontinuous factor.
%Meanwhile we say $(X/\mathcal{A}(Q_{me}),T)$ is the maximal mean equicontinuous fator of $(X,T)$.
\end{lem}

\begin{proof}
Let $Y=X/\mathcal{A}(Q_{me}(X, T))$ and $\pi:X\rightarrow Y$ be the factor map.
As $\pi\colon X\to Y$ is a continuous surjective,
we can choose a metric on $X$ (also denoted by $d$)
such that $d(x,y)\geq d(\pi(x),\pi(y))$  for all $x,y\in X$.

Assume that $(Y,T)$ is not mean equicontinuous.
By Lemma ~\ref{ME-Q}, there exist $x,y\in Y$ with $x\neq y$ and $(x,y)\in Q_{me}(Y,T)$.
Let $\tau<\frac{1}{4}d(x,y)$.
For $k\in \N$, there are $x_k,y_k\in Y$ and $n_k\in \Z$ with $d(x_k,y_k)<\frac{1}{k}$ such that
$$
\frac{1}{n_k}\#(\{0\leq i \leq n_k-1:d(T^ix_k,x)<\tau,d(T^iy_k,y)<\tau\})>c
$$
for some $c>0$.
For every $k\in \N$,  choose $u_k,v_k\in X$ such that $\pi(u_k)=x_k,\pi(v_k)=y_k$.
Without loss of generality, we can assume that
$\lim_{k\rightarrow\infty}x_k=z=\lim_{k\rightarrow\infty}y_k$
and
$\lim_{k\rightarrow\infty}u_k=w=\lim_{k\rightarrow\infty}v_k$,
then $\pi(w)=z$.
$$\frac{1}{n_k}\sum_{i=0}^{n_k-1}d(T^iu_k,T^iv_k)\geq
\frac{1}{n_k}\sum_{i=0}^{n_k-1}d(T^ix_k,T^iy_k)\geq
c\cdot(d(x,y)-2\tau)>\frac{c}{2}\cdot d(x,y)>0.$$
Let $\mu_k=\frac{1}{n_k}\sum_{i=0}^{n_k-1}\delta_{(T^iu_k,T^iv_k)}$
and  $\nu_k=\frac{1}{n_k}\sum_{i=0}^{n_k-1}\delta_{(T^ix_k,T^iy_k)}$.
Without loss of generality, assume that $\mu_k\rightarrow \mu$
and $\nu_k\rightarrow \nu$, where $\mu\in M(X\times X,T\times T)$
and $\nu\in M(Y\times Y,T\times T)$.
From the proof of Theorem ~\ref{general}, we have
$\mu(\text{supp}(\mu)\setminus \Delta_X)>0$ and $\nu(\text{supp}(\nu)\setminus \Delta_Y)>0$.
Moreover $\pi\times \pi(\mu)=\nu$ since $\pi\times \pi(\mu_k)=\nu_k$.

If $(a,b)\in \text{supp}(\mu) $, by the proof of Theorem ~\ref{general}, we have
$(a,b)\in Q_{me}(X,T)$ which implies $\pi(a)=\pi(b)$, hence $\pi\times \pi (\text{supp}(\mu))=\Delta_Y$.

Therefore $\nu(\Delta_Y)=\pi\times \pi(\mu)(\Delta_Y)=\mu((\pi\times \pi)^{-1}\Delta_Y)=
\mu(\text{supp}(\mu))=1$.

It is a contradiction which implies the lemma.
\end{proof}

By Lemma \ref{ME-Q} and Lemma \ref{mean equicontinuous factor}, we have the main result in this section.

\begin{thm} \label{MAX ME}
Let $(X,T)$ be a dynamical system. Then
 $S_{me}(X, T)$ is  the smallest closed $T\times T$ invariant equivalence relation containing $Q_{me}(X, T)$.
\end{thm}

\medskip
\noindent {\bf Acknowledgments.}
The authors would like to thank Jian Li and Xiangdong Ye for bringing us the questions and for useful discussions when doing the research.
We also thank Jie Li for the careful reading which help the writing of the paper. Finally the authors thank the referee for his/her careful reading.

The authors were supported by NNSF of China (11431012).

\bibliographystyle{amsplain}

\begin{thebibliography}{99}

\bibitem{A59} J. Auslander, \textit{Mean-$L$-stable systems},
Illinois J. Math. \textbf{3} (1959), 566--579.


\bibitem{DG16}
T. Downarowicz and E. Glasner, {\it Isomorphic extensions and applications, Topol}, Methods Nonlinear Anal. {\bf 48} (2016), 321--338.

%Downarowicz T, Glasner E. Isomorphic extensions and applications[J]. Topological Methods in Nonlinear Analysis, 2016, 48(1): 321-338.



\bibitem{EG60} R. Ellis and W. H. Gottschalk, \emph{Homomorphisms of transformation groups}, Trans. Amer.
Math. Soc. \textbf{94} (1960), 258--271.

\bibitem{F51} S. Fomin, \emph{On dynamical systems with a purely point spectrum}, Dokl. Akad. Nauk SSSR,
vol. {\bf 77} (1951), 29--32 (In Russian).

%Fomin S. On dynamical systems with a purely point spectrum[C]//Dokl. Akad. Nauk SSSR. 1951, 77(1): 22-32.

\bibitem{F81} H. Furstenberg, \emph{Recurrence in Ergodic Theory and Combinatorial Number Theory},
Princeton Univ. Press, Princeton, NJ, 1981.

%Furstenberg H. Recurrence in ergodic theory and combinatorial number theory[M]. Princeton University Press, 2014.


\bibitem{Felipe1} F. Garcia-Ramos, \textit{ A characterization of $\mu$-equicontinuity for topological dynamical systems},
Proc. Amer. Math. Soc.  {\bf 145} (2017),  no. 8, 3357--3368.



\bibitem{Felipe2} F. Garcia-Ramos, \textit{Weak forms of topological and measure-theoretical equicontinuity: relationships with discrete spectrum and sequence entropy},
Ergodic Theory Dynam. Systems  {\bf 37}  (2017),  no. 4, 1211--1237.


\bibitem{Felipe3} F. Garcia-Ramos, L. Jin, \textit{Mean proximality and mean Li-Yorke chaos. Proc. Amer. Math. Soc},
{\bf 145}  (2017),  no. 7, 2959--2969.

%Garcia-Ramos F, Jin L. Mean proximality and mean Li-Yorke chaos[J]. Proceedings of the American Mathematical Society, 2017, 145(7): 2959-2969.

\bibitem{GLZ17} F. Garc\'{i}a-Ramos, J. Li, R. Zhang,
{\it When is a dynamical system mean sensitive?},
Ergodic Theory Dynam. Systems, to appear (on line in 2017) .

\bibitem{HL}
 W. Huang, J. Li, J. Thouvenot, L. Xu and X. Ye, {\it Mean equicontinuity, bounded complexity and discrete spectrum},
arXiv:1806.02980.

\bibitem{LT14} Jian Li and S. Tu, \emph{On proximality with Banach density one},
J. Math.Anal.Appl. \textbf{416} (2014), 36--51.

\bibitem{YL} Jian Li, S. Tu and X. Ye, {\it Mean equicontinuity and mean sensitivity}, Ergod. Th. Dynam. Sys. {\bf 35} (2015), 2587--2612.

\bibitem{Li} J. Li, \textit{How chaotic is an almost mean equicontinuous system}? DCDS-A., {\bf 38} (2018), no. 9, 4727--4744.
%Li J. How chaotic is an almost mean equicontinuous system[J]. Discrete & Continuous Dynamical Systems-A, 2018, 38(9): 4727-4744.


\bibitem{O52} J. C. Oxtoby, \emph{Ergodic sets}, Bull. Amer. Math. Soc., {\bf 58} (1952), 116--136.


\bibitem{S82} B. Scarpellini,
\emph{Stability properties of flows with pure point spectrum},
J. London Math. Soc. (2)  {\bf 26} (1982), no. 3, 451--464.



\bibitem{W82} P. Walters, \emph{An introduction to ergodic theory}.
Graduate Texts in Mathematics, 79. Springer-Verlag,
New York-Berlin, 1982.

\bibitem{YZ} X. Ye and R. Zhang, \emph{On sensitivity sets
in topological dynamics}, Nonlinearity, {\bf 21} (2008), 1601--1620.


%\bibitem{LJ} F. Garc\'{\i}a-Ramos, J. Li and R. Zhang, When is a dynamical system
%mean sensitive? \emph{Ergodic Theory Dynam. Systems}, 2017, \arXiv{1708.01987}, DOL:1017/etds.2017.101


\end{thebibliography}

\end{document}